\documentclass[11pt,reqno]{amsart}

\usepackage[T1]{fontenc}
\usepackage{lmodern}
\usepackage{microtype}
\usepackage{amsmath,amssymb,amsthm,mathtools}
\usepackage{comment}
\usepackage{dsfont}
\usepackage{enumitem}
\usepackage[margin=1in]{geometry}
\usepackage{xcolor}
\usepackage[colorlinks=true,linkcolor=blue!55!black,citecolor=blue!55!black,urlcolor=blue!55!black]{hyperref}

\numberwithin{equation}{section}

\newtheorem{theorem}{Theorem}[section]
\newtheorem{question}[theorem]{Question}

\theoremstyle{remark}

\newcommand{\R}{\mathbb R}
\newcommand{\G}{\mathbb G}
\newcommand{\Z}{\mathbb Z}
\newcommand{\N}{\mathbb N}

\title[Duality in Harmonic Analysis and beyond]{Duality in Harmonic Analysis and beyond}

\author[Bartosz Langowski]{Bartosz Langowski}
\address{Department of Mathematics and Physical Sciences, Franciscan University of Steubenville, 
1235 University Blvd., Steubenville, OH 43952, USA}
\email{blangowski@franciscan.edu}

\author[Mariusz Mirek]{Mariusz Mirek}
\address{Department of Mathematics, Rutgers University, Piscataway, NJ 08854, USA; and Instytut Matematyczny, Uniwersytet Wroc\l awski, Plac Grunwaldzki 2, 50-384 Wroc\l aw, Poland}
\email{mariusz.mirek@rutgers.edu}

\author[James Wright]{James Wright}
\address{School of Mathematics and Maxwell Institute for Mathematical Science, James Clerk Maxwell Building, The King’s Buildings, Peter Guthrie Tait Road, Edinburgh, EH9 3FD, UK}
\email{J.R.Wright@ed.ac.uk}

\author[B{\l}a{\.z}ej Wr{\'o}bel]{B{\l}a{\.z}ej Wr{\'o}bel}
\address{Instytut Matematyczny Polskiej Akademii Nauk,  \'Sniadeckich 8, 00-656 Warszawa, Poland; and Instytut Matematyczny, Uniwersytet Wroc\l awski, Plac Grunwaldzki 2, 50-384 Wroc\l aw, Poland}
\email{bwrobel@impan.pl}

\thanks{Bartosz Langowski was supported by the AMS–Simons Research Enhancement Grant for Primarily Undergraduate Institution (PUI) Faculty and by the National Science Centre of Poland under the OPUS grant no. 2025/59/B/ST1/01786. Mariusz Mirek was partially supported by the NSF CAREER grant (DMS-2236493).  
B{\l}a{\.z}ej Wr{\'o}bel was supported by the National
Science Centre, Poland, grant Sonata Bis 2022/46/E/ST1/00036. For the purpose of Open Access, the authors have applied a
CC-BY public copyright licence to any Author Accepted Manuscript (AAM) version arising from this submission.}

\date{}

\begin{document}

\begin{abstract}
In this short note, we develop a duality-based, purely
operator-theoretic method for proving several classical maximal
inequalities on $L^p$ spaces for $p\in(1,\infty)$. The method applies
to maximal ergodic operators, Hardy--Littlewood maximal functions, and
maximal Radon averaging operators, and derives their boundedness
directly from structural properties of the underlying operators.
Further, we show that the same duality principle extends beyond
maximal functions. We develop a theory of $r$-variational and jump
inequalities for sequences of bounded positive monotone operators on
Hilbert spaces, using soft arguments that apply to various classical
operators in analysis.  In particular, answering a question that had
been open for some time, we give direct proofs of $r$-variational and
jump inequalities for convolution operators associated with lacunary
dilations of smooth bumps on $\mathbb R^d$, completely avoiding any
appeal to martingale inequalities.
\end{abstract}

\maketitle

\section{Introduction}
The aim of this note is to present a purely operator-theoretic
approach to several classical maximal inequalities, as well as to
variational and even jump inequalities. The arguments developed below
avoid stopping times, covering lemmas, Calder{\'o}n--Zygmund
decompositions, Littlewood--Paley square functions, and measurable
linearizations of the underlying seminorms. Our goal is to highlight the
operator structure behind these estimates and to show that, in a
number of classical situations, this structure alone is sufficient to
obtain the desired bounds. For instance, in the context of maximal
estimates, the key mechanism is to dualize the finite supremum and
retain the resulting expression in vector-valued form.  Namely, let
$(X, \mathcal B(X), \mu)$ be a $\sigma$-finite measure space, let
$p\in(1, \infty)$ and $p':={p}/{(p-1)}$, and let $\mathbb I$ be a
finite set.  For measurable functions $(F_i)_{i\in \mathbb I}$, we
have the identity
\begin{equation}\label{eq:exact-duality}
 \big\|\sup_{i\in \mathbb I}|F_i|\big\|_{L^p(X)}
 =
 \sup_{\big\|\sum_{i\in \mathbb I}|g_i|\big\|_{L^{p'}(X)}\leq 1}
 \Big|\sum_{i\in \mathbb I}\int_X F_i g_id\mu\Big|
\end{equation}
by duality. Consequently, if $S_i:L^p(X)\to L^p(X)$ is a linear operator for every $i\in\mathbb I$, and
\begin{equation}\label{eq:vector-adjoint}
 \Big\|\sum_{i\in \mathbb I}S_i^*g_i\Big\|_{L^{p'}(X)}
 \leq C\Big\|\sum_{i\in \mathbb I}|g_i|\Big\|_{L^{p'}(X)},
\end{equation}
then
\begin{equation}\label{eq:max-from-vector}
 \big\|\sup_{i\in \mathbb I}|S_i f|\big\|_{L^p(X)}\leq C\|f\|_{L^p(X)}.
\end{equation}
In fact, the optimal constants in \eqref{eq:vector-adjoint} and
\eqref{eq:max-from-vector} are equal, since we have
\begin{align}
\label{eq:2}
 \sup_{\|f\|_{L^p(X)}\leq1}
       \big\|\sup_{i\in \mathbb I}|S_i f|\big\|_{L^p(X)}
 =
 \sup_{\big\|\sum_{i\in \mathbb I}|g_i|\big\|_{L^{p'}(X)}\leq1}
       \Big\|\sum_{i\in \mathbb I}S_i^*g_i\Big\|_{L^{p'}(X)}.
\end{align}
Thus the maximal operator and the dualized sum have exactly the same
norm. The last identity, as illustrated in Section \ref{sec:dvj}, has
analogues for the $r$-variational seminorm \eqref{eq:6} and even for
the jump-counting function \eqref{eq:7}; see identities
\eqref{eq:variation-exact-duality} and \eqref{eq:jump-exact-duality},
respectively. This extends the duality method from maximal bounds to
$r$-variational and jump inequalities, which quantify convergence by
measuring fluctuations across scales and transform these nonlinear
quantities into vector-valued expressions involving the original
operators or their adjoints.

In \cite{Bou}, Bourgain used variation and jump estimates to study
quantitative forms of pointwise convergence. His work initiated a
substantial literature in ergodic theory \cite{jkrw, Kr, KMT, Mes,
MSS, MST2, MSZ3, Trojan}, discrete harmonic analysis \cite{IMMS, MSZ1,
MT, RW}, and classical harmonic analysis \cite{JSW, JG, KLMP, MST1,
MSZ2, OSTTW}. A recurring feature of proofs in this literature is the
use of L{\'e}pingle's $r$-variational inequality \cite{Le} for
martingales when $r>2$, or the endpoint $r=2$ jump inequality of
Pisier and Xu \cite{PX}, which was also rediscovered by Bourgain
\cite{Bou} and further extended in \cite{jkrw, MSZ1}. Even for convolution operators associated with dilations of smooth bumps, the $r$-variation and jump inequalities require comparison with the corresponding martingale estimates; see \cite[Section 4]{JSW}. Understanding the mechanism behind these bounds, and in particular whether martingale methods are genuinely necessary, has intrigued many authors since Bourgain's work \cite{Bou} and led to the following natural question:

\begin{question}\label{q:direct-variation-jump}
Let $\varphi:\mathbb R\to \mathbb R$ be a nonnegative Schwartz
function with mean one, and assume that its Fourier transform is
nonnegative even and decreasing on $\mathbb R_+:=(0,\infty)$.  For a
lacunary sequence $(t_n)_{n\in\mathbb Z}\subseteq\mathbb R_+$, set
$T_nf:=\varphi_{t_n}*f$, where $\varphi_t(x)=t^{-1}\varphi(x/t)$. The
question is whether one can prove
\begin{align*}
\|V^r(T_nf:n\in\mathbb Z)\|_{L^2(\mathbb R)}
&\lesssim_r\|f\|_{L^2(\mathbb R)}, \qquad r>2,\\
\sup_{\lambda\in\mathbb R_+}
\|\lambda N_\lambda(T_nf:n\in\mathbb Z)^{1/2}\|_{L^2(\mathbb R)}
&\lesssim\|f\|_{L^2(\mathbb R)},
\end{align*}
without appealing to the corresponding martingale inequalities.
\end{question}

Theorem \ref{thm:abstract} gives an affirmative answer to Question
\ref{q:direct-variation-jump} within an abstract framework. We
establish maximal, $r$-variational, and jump inequalities for
positivity-preserving operators on Hilbert spaces under suitable
monotonicity and positivity assumptions. This is somewhat surprising,
since for a long time the projectivity property of conditional
expectations induced by a filtration was viewed as a decisive
mechanism behind $r$-variational and jump inequalities for
martingales. Theorem \ref{thm:abstract} shows that, in fact, these
estimates can arise from a different set of structural properties.

\subsection{A brief history and  overview of the dualization method} The use of
duality to study maximal inequalities has a long history, particularly
in martingale theory. The dual form of Doob's inequality appears in
the work of Burkholder, Davis, and Gundy \cite{BDG}; see also
\cite{AMS,Dil} for related formulations and arguments. This point of
view is central to Junge's proof of the noncommutative Doob inequality
\cite{J} and is developed further in the theory of noncommutative
maximal ergodic inequalities \cite{JX}. The idea of dualizing maximal
functions was also extensively exploited by Bourgain in \cite{Bou} in
the study of various maximal inequalities in ergodic theory and
harmonic analysis, including his celebrated logarithmic maximal
inequality, see \cite[Sections 4 and 7]{Bou}.

Our starting
point is the duality identity in \eqref{eq:2}. We examine it first in
the context of maximal ergodic averages, see Theorem
\ref{thm:ergodic-maximal}, and then for maximal convolution operators
associated with sequences of Radon probability measures on locally
compact abelian groups arising naturally in harmonic analysis, see
Theorem \ref{thm:main}. These include Hardy--Littlewood averages,
Radon-type averaging operators as well as lacunary spherical averages in $\R^d$, see \cite{DF, bigs} and \cite{MSZ2}. The substantive question is how to
estimate the sum on the right-hand side of \eqref{eq:2} uniformly in
the number of operators. Duality alone does not provide such an
estimate. In the martingale setting, conditional expectations $(\mathbb E_k)_{k\in\Z}$ satisfy
the projection property $\mathbb E_k\mathbb E_n=\mathbb E_{\min\{k, n\}}$. This identity provides the structure
needed to control interactions between different indices. For the
families considered in Theorems \ref{thm:ergodic-maximal} and
\ref{thm:main}, different structural properties serve this purpose:
\begin{itemize}
\item[(i)] For the ergodic averages in \eqref{eq:ergodic-average}, invariance and an integrated comparison between scales lead to the so-called engulfing inequality \eqref{eq:ergodic-engulfing}, which in turn yields a moment estimate providing the required uniformity with respect to the number of underlying operators.

\item[(ii)] For the convolution operators $T_n$ in \eqref{eq:Tn-def}, the low and high-frequency Fourier estimates given respectively in \eqref{eq:low-frequency} and \eqref{eq:high-frequency} control the interactions between separated scales. The decisive space here is $L^2$, where Plancherel's theorem allows the low-frequency approximation of one multiplier to be combined directly with the high-frequency decay of the other, producing the required cancellation between separated scales.

\item[(iii)] To reach the full $L^p$ range, an averaged translation estimate controls the failure of the identity $T_kf\,T_ng=T_k(fT_ng)$. More precisely, the difference $T_kf\,T_ng-T_k(fT_ng)$  is small, and this smallness yields maximal estimates uniformly in the number of operators.
\end{itemize}

The emphasis here is on estimating the resulting vector-valued sums arising in \eqref{eq:2} directly from properties of the underlying operators. For the families considered in this paper, invariance, positivity, and suitable estimates for interactions between different scales lead to inequalities involving the same finite vector norms that we ultimately seek to bound.

In Theorem \ref{thm:abstract}, presented in Section \ref{sec:hil}, we establish $r$-variational and jump inequalities on $L^2$ for sequences of bounded positive monotone  operators. The main point is a pair of exact dualization identities for $r$-variations, \eqref{eq:variation-exact-duality}, and jump-counting functions, \eqref{eq:jump-exact-duality}. As in the case of maximal functions, these identities transform the nonlinear variational and jump quantities into vector-valued expressions involving the original operators and their increments. After dualization, self-adjointness and telescoping reduce the problem to a single Hilbert-space differencing estimate.

The proof of Theorem \ref{thm:abstract} is entirely based on soft Hilbert-space arguments: the Cauchy--Schwarz inequality for positive operators, positivity preservation, monotonicity and telescoping identities. In particular, the argument uses neither martingale inequalities nor martingale difference decompositions, and it requires no commutativity, convolution structure, smoothness, dilation structure, interpolation, transference, or Calder{\'o}n--Zygmund theory. This abstract approach applies directly to anisotropic convolution operators, ordered Radon convolution measures on locally compact abelian groups, and conditional expectations associated with decreasing filtrations.

A key feature of the method is that, despite the different settings involved, the proofs follow the same sequence of steps. First, we dualize the seminorm appearing inside the $L^p$ norm: for maximal functions we invoke \eqref{eq:2}, while for $r$-variations and jump inequalities we use the dualization formulas \eqref{eq:variation-exact-duality} and \eqref{eq:jump-exact-duality}, respectively. In each case, the nonlinear quantity is transformed into an $\ell^1$-valued sum. This avoids introducing a measurable choice of scale, so that one works with the original operators themselves or with their adjoints. Next, we expand an appropriate norm or moment of this sum. We then isolate the interaction that captures the favorable behavior of the operators and constitutes the main term. The remaining terms are estimated separately. Those that can be controlled directly by the available operator estimates are absorbed into the error term, whereas the terms that still involve the unknown maximal quantity give rise to a bootstrap inequality. Finally, we close the argument by induction on the size of the operator norm.

Our purpose is not to advocate dualization as a replacement for
classical methods. In particular, the present approach appears less
directly adapted to weak-type $(1,1)$ estimates. Rather, we view it as
a complementary method that isolates the underlying operator structure
of the problem. This perspective is particularly natural in ergodic
theory, where averaging operators are often studied on Hilbert spaces
and spectral or operator-theoretic methods are decisive. It may be
especially useful in settings involving actions of more general
groups \cite{Nevo}, where geometric tools such as covering arguments are
unavailable or poorly adapted, while Hilbert-space techniques remain
accessible. As we shall see below,
the dualization identity \eqref{eq:2} leads to a simple operator-theoretic framework with
consequences that are not immediately apparent from the standard
proofs.

\subsection{Notation} Here we gather the basic notation that will be used throughout the paper. The sets $\mathbb Z$, $\mathbb R$, $\mathbb C$, and $\mathbb T:=\mathbb R/\mathbb Z$ have their standard meanings. The sets of nonnegative and positive integers are denoted, respectively, by $\mathbb N:=\{0,1,2,\ldots\}$ and $\mathbb Z_+:=\{1,2,\ldots\}$. We also write $\mathbb R_+:=(0,\infty)$. For any $N\in\mathbb R_+$, we define $[N]:=(0,N]\cap\mathbb Z$ and $[\pm N]:=[-N,N]\cap\mathbb Z$. All vector spaces are defined over $\mathbb C$. Unless otherwise stated, $(X,\mathcal B(X),\mu)$ denotes a $\sigma$-finite measure space. For $p\in(0,\infty)$, the space of all measurable functions $f:X\to\mathbb C$ such that $|f|^p$ is integrable is denoted by $L^p(X)$, whereas $L^\infty(X)$ denotes the space of all essentially bounded measurable functions.

\section*{Acknowledgments}
We have been working on this project since September 2025. We first
proved Theorem \ref{thm:ergodic-maximal} in the martingale setting,
where, using a numerical inequality 
\cite[Lemma 2.3]{IW}, we established a variant of the engulfing
inequality \eqref{eq:ergodic-engulfing}. At that time, we were unaware
of the work of Burkholder, Davis, and Gundy \cite{BDG}, as well as
\cite{AMS,Dil}, although we were aware that duality had been
extensively exploited by Bourgain in \cite{Bou}. We then extended this
approach to the Hardy--Littlewood maximal operator and to the maximal
ergodic operator, essentially as in Theorem \ref{thm:ergodic-maximal}.  As a next
step, we established Theorem \ref{thm:main} for $p=2$ starting with the following identity
\[
\Big\|\sum_{n\in\mathbb I}T_nf_n\Big\|_{L^2(\mathbb G)}^2
=
\sum_{n\in\mathbb I}\|T_nf_n\|_{L^2(\mathbb G)}^2
+
2\operatorname{Re}
\sum_{\substack{n,k\in\mathbb I\\ n<k}}
\langle f_n,T_kf_k\rangle+2\operatorname{Re}
\sum_{\substack{n,k\in\mathbb I\\ n<k}}
\langle T_nf_n-f_n,T_kf_k\rangle.
\]
The first sum was handled using the boundedness of the operators, the second by positivity together with a best-constant argument, while the last one was estimated using Plancherel's theorem and  inequalities \eqref{eq:low-frequency} and \eqref{eq:high-frequency}. Then, we closed the argument by bootstrapping in the size of the resulting constants.
Subsequently, we completed the $L^p$ theory by observing that, in order to extend the $L^2$ argument beyond the Hilbert-space setting, the term appearing under the inner product in the last sum should be further decomposed as
\begin{align}
\label{eq:9}
(T_nf_n-f_n)T_kf_k
=
(T_n-I)(f_nT_kf_k)
+
T_nf_n\,T_kf_k
-
T_n(f_nT_kf_k).
\end{align}
This decomposition separates the part that retains the $L^2$
cancellation from the commutator-type term needed to pass to the full
$L^p$ range.  We finally began extending the duality method beyond
maximal functions and first proved Theorem \ref{thm:abstract} for
convolution operators associated with dilations of smooth bumps, as in
Question \ref{q:direct-variation-jump}, before arriving at its present
abstract formulation.

While preparing the first draft of the paper in September 2026, we
began experimenting with AI. Initially, it was unable to prove Theorem
\ref{thm:main} even for $p=2$. It was successful only for measures
satisfying an engulfing-type inequality such as
\eqref{eq:ergodic-engulfing}. For instance, for Radon measures
supported on the parabola, it consistently reduced the problem to the
classical approach involving the Hardy--Littlewood maximal inequality
and Littlewood--Paley theory.  After being provided with our notes
containing the $L^2$ theory, however, it was able to extend the
argument to the $L^p$ setting.  AI simplified our argument by eliminating one term from the original decomposition in \eqref{eq:9}, reducing the analysis to the expression $T_nf_n\,T_kf_k-T_n(f_nT_kf_k)$ and allowing us to present the $L^2$ and $L^p$ theories in a unified way. We also used AI to prepare an initial
draft of these notes from material written by the authors. AI was used
to make all constants explicit, reduce the number of quantifiers in
the statements of our theorems, and assist in identifying the
maximizing functions $w_\pi^{(u)}$ used to establish formula
\eqref{eq:jump-exact-duality}. AI was also useful in formulating
Theorem \ref{thm:abstract} in its present abstract form.  All
mathematical ideas, strategies, and underlying intuitions were
developed by the authors. The authors are  responsible for the
final version of the paper and for the correctness of all its
arguments.\section{Ergodic averages}

Let $(X,\mathcal B(X),\mu)$ be a $\sigma$-finite measure space and let $T:X\to X$ be an
invertible measure-preserving transformation.  Define the symmetric ergodic
averages
\begin{equation}\label{eq:ergodic-average}
 A_N f(x):=\frac{1}{2N+1}\sum_{n\in[\pm N]}f(T^nx),
 \qquad x\in X, \ N\in\N.
\end{equation}
They are self-adjoint,
positivity preserving, contractions on every $L^p(X)$. 

\begin{theorem}[Maximal ergodic theorem]\label{thm:ergodic-maximal}
For every $p\in(1,\infty]$, there exists a constant $C_p\in\mathbb R_+$ such that for every $f\in L^p(X)$, we have
\begin{equation}\label{eq:ergodic-maximal}
 \big\|\sup_{N\in \N}|A_Nf|\big\|_{L^p(X)}
 \leq C_p\left\|f\right\|_{L^p(X)}.
\end{equation}
\end{theorem}

The proof of Theorem \ref{thm:ergodic-maximal} applies equally well to
martingales \cite{AMS, BDG, Dil, J}, Hardy--Littlewood maximal functions in both the
continuous and discrete settings \cite{bigs}, and many other situations \cite{JX, W, YK}.  The
maximal inequality in Theorem \ref{thm:ergodic-maximal}, together with
convergence on a dense subclass, implies the almost everywhere
convergence of the averages \eqref{eq:ergodic-average} for functions
in $L^p(X)$ with $p\in(1,\infty)$. This recovers Birkhoff's celebrated
pointwise ergodic theorem \cite{Bir} in this range.  It is well known that the
corresponding maximal operator also satisfies a weak-type $(1,1)$
inequality \cite{W, YK}. Combined with convergence on a dense subclass, this
extends the pointwise convergence theorem to all functions in
$L^1(X)$. The dualization method does not seem well suited to
weak-type $(1,1)$ estimates, and an endpoint analogue remains unclear.

From the point of view of ergodic theory and many of its statistical
applications, ergodic theorems are most naturally considered on finite
measure spaces and for bounded functions. Under an ergodicity
assumption, the pointwise ergodic theorem identifies the limit of the
averages \eqref{eq:ergodic-average} with the space average. In the
infinite-measure setting, the ergodic averages of integrable functions
converge to zero under the usual ergodicity assumptions. Since bounded
functions are sufficient for the vast majority of such applications,
establishing \eqref{eq:ergodic-maximal} on $L^2(X)$ is already more
than enough whenever $\mu(X)<\infty$, since in this case
$L^\infty(X)\subseteq L^2(X)$. At the same time, precisely in the
$L^2(X)$ setting, a purely operator-theoretic approach to maximal
functions may be decisive in problems that go beyond the classical
framework, see especially \cite{Nevo} and the references therein.

\begin{proof}[Proof of Theorem \ref{thm:ergodic-maximal}]
Fix an integer $k\geq2$ and define the conjugate exponent $p_k:=k/(k-1)$. For any finite set \(\mathbb I=\{N_i:i\in[L]\}\subset\mathbb N\) with $L\in\Z_+$, it suffices to prove that
\begin{align}
\label{eq:1}
\big\|\sup_{i\in[L]}|A_{N_i}f|\big\|_{L^{p_k}(X)}
\leq k2^{k-1}\|f\|_{L^{p_k}(X)}.
\end{align}
Then \eqref{eq:ergodic-maximal} for \(p=p_k\) follows from \eqref{eq:1} and the monotone convergence theorem, and interpolation yields \eqref{eq:ergodic-maximal} for every \(p\in(1, \infty]\), since $\lim_{k\to\infty}p_k=1$. Using \eqref{eq:2}, we will work with the dualized formulation of \eqref{eq:1}.
For an integer $k\geq2$ and $L\in\Z_+$, let $B_{k,L}\ge0$ be the best constant in
\begin{equation}\label{eq:ergodic-vector}
 \Big\|\sum_{i\in[L]}A_{N_i}h_i\Big\|_{L^k(X)}
 \leq B_{k,L}\Big\|\sum_{i\in[L]}|h_i|\Big\|_{L^k(X)},
\end{equation}
uniformly over $N_1,\ldots,N_L\in\N$.  It is enough to consider
$h_i\geq0$.  The constant $B_{k,L}$ is finite  by the triangle inequality  (the bound $B_{k,L} \le L$ trivially holds). For every integer $k\geq2$ and every $L\in\Z_+$, we prove that
\begin{equation}\label{eq:ergodic-B-bound}
 B_{k,L}\leq k2^{k-1}.
\end{equation}
For this purpose we will proceed in two steps.
\smallskip
\paragraph{\textbf{Step 1}} The key observation is the following engulfing inequality.
Let $k\geq2$, let $0\leq M\leq N_j$ for $2\leq j\leq k$, and let
$f_1,\ldots,f_k$ be nonnegative functions.  Then
\begin{equation}\label{eq:ergodic-engulfing}
 J:=\int_X A_Mf_1\prod_{j=2}^{k}A_{N_j}f_j\,d\mu
 \leq
 2^{k-1}\int_X f_1\prod_{j=2}^{k}A_{2N_j}f_j\,d\mu.
\end{equation}
Expanding only the first average and using invariance under $T$ gives
\begin{align*}
J=\frac{1}{2M+1}\sum_{m\in [\pm M]}
 \int_X f_1(y)\prod_{j=2}^{k}A_{N_j}f_j(T^{-m}y)\,d\mu(y).
\end{align*}
If $|m|\leq M\leq N$, then
\begin{align*}
A_Nf(T^{-m}y) \leq\frac{4N+1}{2N+1}A_{2N}f(y)
\leq2A_{2N}f(y).
\end{align*}
Inserting this bound to $J$, inequality   \eqref{eq:ergodic-engulfing} follows.
\smallskip
\paragraph{\textbf{Step 2}} Set
\[
F:=\sum_{i\in[L]}A_{N_i}h_i, \qquad H:=\sum_{i\in[L]}h_i, \qquad \widetilde F:=\sum_{i\in[L]}A_{2N_i}h_i.
\]
Expanding the full $k$-th moment, we have 
\begin{equation}\label{eq:full-moment}
 \int_X F^k\,d\mu
 =\sum_{i_1,\ldots,i_k\in[L]}
 \int_X\prod_{r=1}^{k}A_{N_{i_r}}h_{i_r}\,d\mu.
\end{equation}
In each ordered $k$-tuple choose one position $r$ for which $N_{i_r}$ is
minimal. Invoking \eqref{eq:ergodic-engulfing}, we obtain
\[
\int_X\prod_{s=1}^{k}A_{N_{i_s}}h_{i_s}\,d\mu \leq2^{k-1} \int_Xh_{i_r}\prod_{s\neq r}A_{2N_{i_s}}h_{i_s}\,d\mu.
\]
Using this and \eqref{eq:full-moment}, applying  H\"older's inequality and using \eqref{eq:ergodic-vector} for $\widetilde F$, we may conclude that
\begin{align*}
\|F\|_{L^k(X)}^k
\leq k2^{k-1}\|H\|_{L^k(X)}
\|\widetilde F\|_{L^k(X)}^{k-1}
\leq k2^{k-1}B_{k,L}^{k-1}\|H\|_{L^k(X)}^{k},
\end{align*}
which in turn implies that  $ B_{k,L}^{k}\leq k2^{k-1}B_{k,L}^{k-1}$ and 
\eqref{eq:ergodic-B-bound} follows as desired.
\end{proof}

\section{Convolutions with Radon measures} 
\label{sec:radon}
Let $\G=(\G, +)$ be a locally compact abelian (LCA) group, written additively,
with Haar measure $dx$.  Its Pontryagin dual $\widehat\G$ consists of 
continuous characters $\chi:\G\to\mathbb T$.  Write $\mathbf e$ for the
trivial character and choose the dual Haar measure $d\chi$ with the Plancherel
normalization.  Our conventions are
\[
 \widehat f(\chi):=\int_\G f(x)\overline{\chi(x)}\,dx,
 \qquad
 \widehat\mu(\chi):=\int_\G\overline{\chi(y)}\,d\mu(y),
 \qquad
\|\widehat f\|_{L^2(\widehat\G)}=\left\|f\right\|_{L^2(\G)}.
\]
The Fourier transform $\widehat f$ is initially defined for $f\in L^1(\G)$; on $L^2(\G)$
we use its unitary extension.  These are the standard group
Fourier conventions; see, for example, \cite{Rudin}.
Let $(\mu_n)_{n\in\Z}$ be Radon probability measures on $\G$ and let
$(s_n)_{n\in\Z}$ be positive numbers satisfying
\begin{equation}\label{eq:lacunary-scales}
\inf_{n\in\Z} \frac{s_{n+1}}{s_n}=\lambda>1.
\end{equation}
For each $n\in \Z$ any $x\in\G$ and $\chi\in\widehat{\G}$, set
\begin{equation}\label{eq:Tn-def}
 T_nf(x):=f*\mu_n(x)=\int_\G f(x-y)\,d\mu_n(y),
 \qquad m_n(\chi):=\widehat{\mu_n}(\chi).
\end{equation}
Thus $\widehat{T_nf}=m_n\widehat f$, and each $T_n$ is a
positive contraction on every $L^p(\G)$, preserves the  $L^1(\G)$ norm, and
preserves constants on $L^\infty(\G)$.
Fix a measurable function $\rho:\widehat\G\to[0,\infty)$ with
$\rho(\mathbf e)=0$.  Assume that there are $\alpha,\beta\in\R_+$ and $C_0, C_{\infty}\in\R_+$ such that  for every  $n\in\Z$ and $\chi\in\widehat\G$, the following estimates hold:
\begin{align}
|1-m_n(\chi)|&\leq C_0\min\{1,(s_n\rho(\chi))^\alpha\},
\label{eq:low-frequency}\\
|m_n(\chi)|&\leq C_\infty\min\{1,(s_n\rho(\chi))^{-\beta}\}.
\label{eq:high-frequency}
\end{align}
Estimate \eqref{eq:low-frequency} forces $m_n(\chi)=1$ on
$\{\chi\in\widehat{\G}:\rho(\chi)=0\}$, so the second minimum is
defined to be $1$ when $\rho(\chi)=0$.  No triangle inequality,
homogeneity or continuity of $\rho$ is needed.  The value $\rho(\chi)$ is a frequency parameter, not necessarily a norm. No group dilation
or relation among the measures is assumed.  All constants below depend
only on the displayed numerical parameters, not on $\G$ or $\rho$.

\begin{theorem}[Maximal estimates for convolutions with Radon measures]
\label{thm:main}
Let $(\mu_n)_{n\in\mathbb Z}$ be a sequence of Radon probability measures on an LCA group $\mathbb G$ satisfying conditions \eqref{eq:lacunary-scales}--\eqref{eq:high-frequency}. Then, for every $p\in(1,\infty]$, there exists a constant
$C(p):=C(p;\alpha,\beta,\lambda,C_0,C_\infty)\in\mathbb R_+$ such that, for every $f\in L^p(\mathbb G)$, we have
\begin{equation}\label{eq:max}
\big\|\sup_{n\in\Z}|f*\mu_n|\big\|_{L^p(\G)}
\leq C(p)\|f\|_{L^p(\G)}.
\end{equation}
\end{theorem}

Theorem \ref{thm:main} recovers the maximal estimates for polynomial
Radon averaging operators, as well as the maximal estimates for
lacunary spherical averaging operators in $\mathbb R^d$, see
\cite{DF,bigs} and also \cite{MSZ2}. Interestingly, the proof of
Theorem~\ref{thm:main} does not compare $\mu_n$ with a smooth measure nor
dominates $T_n$ by a ball, or Hardy--Littlewood
average, as done in earlier approaches \cite{DF, MSZ2, bigs}. In
particular, the measures need not arise by dilating a single model
measure. The conclusion follows directly from the properties of the
operators themselves, without appealing to Littlewood--Paley theory
as in \cite{DF, MSZ2, bigs}. Nevertheless, $L^2(\mathbb G)$ remains the
decisive analytic space. The required cancellation arises from
combining the low-frequency behavior \eqref{eq:low-frequency} of one
multiplier with the high-frequency decay \eqref{eq:high-frequency} of
another. Fourier analysis is used only at this stage, and simple
interpolation then transports the resulting decay beyond the Hilbert
space $L^2(\mathbb G)$.

An important feature of the proof of Theorem \ref{thm:main} is that it
extends beyond convolution operators. The same method can be applied
to maximal functions associated with ergodic operators whenever the
corresponding averages admit a spectral representation. What
is needed is a representation that allows one to formulate and verify
analogues of the low-frequency approximation estimate
\eqref{eq:low-frequency} and the high-frequency decay estimate
\eqref{eq:high-frequency}. In this setting, the spectral
representation plays the role of the Fourier transform, and the rest
of the argument proceeds in essentially the same way.

\begin{proof}[Proof of Theorem \ref{thm:main}]
We begin by proving that for every exponent $r\in[2, \infty)$ there exists a constant
$C(r):=C(r; \alpha, \beta, \lambda, C_0, C_{\infty})\in\R_+$ such that for
every finite set $\mathbb I\subset\Z$ and every family
$(f_n)_{n\in \mathbb I}\subseteq L^r(\G)$, we have
\begin{equation}
\label{eq:3}
\Big\|\sum_{n\in \mathbb I}T_nf_n\Big\|_{L^r(\G)}
\leq C(r)\Big\|\sum_{n\in \mathbb I}|f_n|\Big\|_{L^r(\G)}.
\end{equation}
Assuming momentarily that \eqref{eq:3} holds, we show how it implies
\eqref{eq:max}.  The adjoint $T_n^*$ is a convolution operator with
the reflected measure $\widetilde\mu_n(E)=\mu_n(-E)$.  Its multiplier
is $\overline{m_n}$, which also satisfies conditions
\eqref{eq:low-frequency} and \eqref{eq:high-frequency}.  Consequently,
inequality \eqref{eq:3} also holds for the adjoint operator $T_n^*$ in place of
$T_n$ and $r=p/(p-1)\in[2, \infty)$ for $p\in(1, 2]$.
Therefore, using \eqref{eq:2}, we obtain
\[
\big\|\sup_{n\in \mathbb I}|T_nf|\big\|_{L^p(\G)} \leq C(p/(p-1))\|f\|_{L^p(\G)}.
\]
Letting $\mathbb I$ increase to $\mathbb Z$, the monotone convergence theorem yields \eqref{eq:max} for all $p\in(1,2]$, and interpolation extends it to the full range $p\in(1,\infty]$, as desired.

We now turn to the proof of \eqref{eq:3}. For this purpose, we introduce the following notation. For $r\in[2,\infty)$, we define the following quantities:
\begin{equation}\label{eq:comm-constants}
\kappa:=\min\{\alpha,\beta\},\qquad
K:=(2C_0C_\infty)^{1/2},\qquad
A_r:=2^{1-2/r}K^{2/r},\qquad
L_r:=\frac{A_r}{\lambda^{\kappa/r}-1}.
\end{equation}

We now fix a finite nonempty set $\mathbb I\subset\Z$ and let $D_{r, \mathbb I}$ be the best constant in \eqref{eq:3}.
The triangle inequality gives $D_{r,\mathbb I}\leq |\mathbb I|$. Positivity allows the operator norm in \eqref{eq:3} to be computed on nonnegative inputs, and the preservation of the $L^1$ norm consequently gives $D_{1,\mathbb I}=1$.

We will show that, for every $r\in[2,\infty)$, the following bootstrap inequality holds:
\begin{align}
\label{eq:4}
D_{r,\mathbb I}^2\leq D_{r/2, \mathbb I}(1+2D_{r, \mathbb I})
+2L_rD_{r-1, \mathbb I}^{(r-1)/r}.
\end{align}

We complete the proof of the theorem in the following four steps.
\smallskip
\paragraph{\textbf{Step 1}}
We show that the following averaged translation estimate holds. Namely, for any $k,n\in\mathbb Z$ with $k<n$, any $r\in[2,\infty)$, and any $g\in L^r(\mathbb G)$, we have
\begin{equation}
\label{eq:averaged-translation}
\left(\int_{\G}
\left\|(I-\tau_y)T_ng\right\|_{L^r(\G)}^r\,d\mu_k(y)\right)^{1/r}
\leq A_r\lambda^{-\kappa(n-k)/r}\|g\|_{L^r(\G)},
\end{equation}
where $\tau_y h(x):=h(x-y)$ denotes the translation operator for $x,y\in\mathbb G$.

Consider the linear operator
\[
R_{n}g(x,y):=(I-\tau_y)T_ng(x), \qquad x,y\in\mathbb G,
\]
on the product measure space ${\mathbb H} := ({\mathbb G}, dx) \otimes ({\mathbb G}, d\mu_k)$.
For $g\in L^2(\mathbb G)$, Plancherel and Fubini give
\begin{align}
\label{eq:averaged-plancherel}
\begin{split}
\left\|R_{n}g\right\|_{L^2({\mathbb H})}^2
&=\int_{\widehat\G}|\widehat g(\chi)|^2|m_n(\chi)|^2
\int_{\G}|1-\overline{\chi(y)}|^2\,d\mu_k(y)\,d\chi\\
&=2\int_{\widehat\G}\operatorname{Re}(1-m_k(\chi))
|m_n(\chi)|^2|\widehat g(\chi)|^2\,d\chi,
\end{split}
\end{align}
since $\widehat{\tau_y h}(\chi)=\overline{\chi(y)}\, \widehat h(\chi)$
and $|1-\overline z|^2=2(1-\operatorname{Re}z)$ for
$z\in\mathbb T$. Combining  \eqref{eq:low-frequency} and
\eqref{eq:high-frequency}, we see that if $k<n$, then
$ |m_k(\chi)-1|\,|m_n(\chi)| \leq C_0C_\infty \lambda^{-\kappa(n-k)}$
for every $\chi\in\widehat\G$. Inserting this bound to \eqref{eq:averaged-plancherel}, we immediately obtain 
\begin{align*}
\left\|R_{n}\right\|_{L^2(\G)\to L^2({\mathbb H})}
\leq K\lambda^{-\kappa(n-k)/2}.
\end{align*}
On the other hand, probability convolution and translation are
$L^\infty(\G)$ contractions, so
\[
\left\|R_{n}\right\|_{L^\infty(\G)\to L^\infty({\mathbb H})}\leq2.
\]
Interpolating between these two bounds, we obtain the desired estimate in \eqref{eq:averaged-translation}.

\smallskip
\paragraph{\textbf{Step 2}}
For $k, n\in\mathbb Z$, define $C_{k,n}(f,g):=T_kf\,T_ng-T_k(f\,T_ng)$, then 
\begin{equation*}
C_{k,n}(f,g)(x)
=\int_{\G}f(x-y)
\bigl[T_ng(x)-T_ng(x-y)\bigr]\,d\mu_k(y), \qquad x\in\G.
\end{equation*}
We prove that for every $r\in[2,\infty)$ and nonnegative family $(f_n)_{n\in \mathbb I}\subseteq L^r(\G)$, we have 
\begin{equation}
\label{eq:comm-total}
\Big\|\sum_{\substack{k,n\in \mathbb I\\k<n}}
C_{k,n}(f_k,f_n)\Big\|_{L^{r/2}(\G)}
\leq L_rD_{r-1, \mathbb I}^{(r-1)/r}\Big\|\sum_{k\in \mathbb I}f_k\Big\|_{L^r(\G)}^2.
\end{equation}
To prove \eqref{eq:comm-total}, we define $H:=\sum_{k\in \mathbb I}f_k$ and $\mathbb I_j:=\{k\in \mathbb I:k+j\in \mathbb I\}$ for $j\in\Z_+$. We can then write
\begin{align*}
\sum_{\substack{k,n\in \mathbb I\\k<n}}
C_{k,n}(f_k,f_n)=\sum_{j\in\Z_+}E_j,
\quad \text{ where } \quad
E_j:=\sum_{k\in\mathbb I_j}C_{k,k+j}(f_k,f_{k+j}).
\end{align*}
It suffices to prove that for every $j\in\Z_+$,
\begin{equation}\label{eq:comm-gap}
\|E_j\|_{L^{r/2}(\G)}
\leq A_r\lambda^{-\kappa j/r}D_{r-1, \mathbb I}^{(r-1)/r}\left\|H\right\|_{L^r(\G)}^2.
\end{equation}
Once \eqref{eq:comm-gap} is established, summing over $j\in\mathbb Z_+$ yields \eqref{eq:comm-total}.
To prove \eqref{eq:comm-gap}, by H\"older's inequality we note that $ |E_j(x)|\leq U_j(x)V_j(x)$, where
\begin{align*}
U_j(x)&:=\Big(\sum_{k\in \mathbb I_j}T_k(f_k^{r'})(x)\Big)^{1/r'},\\
V_j(x)&:=\bigg(\sum_{k\in \mathbb I_j}\int_{\G}
|(I-\tau_y)T_{k+j}f_{k+j}(x)|^r\,d\mu_k(y)\bigg)^{1/r}.
\end{align*}
By the definition of $D_{r-1, \mathbb I}$, we then obtain 
\begin{align*}
\|U_j\|_{L^r(\G)}
=\Big\|\sum_{k\in \mathbb I_j}T_k(f_k^{r'})\Big\|_{L^{r-1}(\G)}^{1/r'}
\leq D_{r-1, \mathbb I}^{1/r'}
\Big\|\sum_{k\in \mathbb I_j}f_k^{r'}\Big\|_{L^{r-1}(\G)}^{1/r'}
\leq D_{r-1, \mathbb I}^{1/r'}\|H\|_{L^r(\G)}.
\end{align*}
By Fubini's theorem and inequality \eqref{eq:averaged-translation}, we have 
\begin{align*}
\|V_j\|_{L^r(\G)}^r
&=\sum_{k\in \mathbb I_j}\int_{\G}
\|(I-\tau_y)T_{k+j}f_{k+j}\|_{L^r(\G)}^r\,d\mu_k(y)\\
&\leq A_r^r\lambda^{-\kappa j}
\sum_{k\in \mathbb I_j}\left\|f_{k+j}\right\|_{L^r(\G)}^r
\leq A_r^r\lambda^{-\kappa j}\|H\|_{L^r(\G)}^r.
\end{align*}
Using H{\"o}lder's inequality  and combining the last two bounds, we obtain \eqref{eq:comm-gap}, as claimed.

\smallskip

\paragraph{\textbf{Step 3}}
We are now ready to prove the bootstrap inequality in \eqref{eq:4}.
By positivity, it suffices to take $f_k\geq0$ for all $k\in\mathbb I$.  Define
$F:=\sum_{k\in \mathbb I}T_kf_k$ and $H:=\sum_{k\in \mathbb I}f_k$.
Expanding $F^2$, using the definition of  $C_{k,n}(f_k,f_n)$ and inequality $(T_kf_k)^2\leq T_k(f_k^2)$, we obtain
\[
F^2\leq\sum_{k\in \mathbb I}T_kb_k+2|E|,
\]
where
\[
b_k:=f_k^2+2f_k\sum_{\substack{n\in \mathbb I\\k<n}}T_nf_n,
\quad \text{ and }\quad
E:=\sum_{\substack{k,n\in \mathbb I\\k<n}}C_{k,n}(f_k,f_n).
\] 
Here $b_k\geq0$ and $\sum_{k\in \mathbb I}b_k\leq H^2+2HF$.  By the definition of $D_{r/2, \mathbb I}$ and $D_{r, \mathbb I}$,
 H\"older's inequality, and estimate \eqref{eq:comm-total}, we conclude that
\begin{align*}
\left\|F\right\|_{L^r(\G)}^2
&\leq D_{r/2, \mathbb I}\bigl(\left\|H\right\|_{L^r(\G)}^2
+2\left\|H\right\|_{L^r(\G)}\left\|F\right\|_{L^r(\G)}\bigr)
+2L_rD_{r-1, \mathbb I}^{(r-1)/r}\left\|H\right\|_{L^r(\G)}^2\\
&\leq
\bigl[D_{r/2, \mathbb I}(1+2D_{r, \mathbb I})+2L_rD_{r-1, \mathbb I}^{(r-1)/r}\bigr]\left\|H\right\|_{L^r(\G)}^2.
\end{align*}
This readily yields the asserted inequality in \eqref{eq:4}.
\smallskip

\paragraph{\textbf{Step 4}} Now we can prove inequality \eqref{eq:3}. Indeed, by a simple interpolation,  we have
\begin{align*}
D_{r/2, \mathbb I}\le D_{r, \mathbb I}^{(r-2)/(r-1)}D_{1, \mathbb I}^{1/(r-1)},
\quad \text{ and }\quad
D_{r-1, \mathbb I}\le D_{r, \mathbb I}^{r(r-2)/(r-1)^2}D_{1, \mathbb I}^{1/(r-1)^2}.
\end{align*}
Using that  $D_{1, \mathbb I}=1$ and inserting these inequalities to \eqref{eq:4}, we obtain
\begin{align*}
D_{r,\mathbb I}^2\leq D_{r, \mathbb I}^{(r-2)/(r-1)}(1+2D_{r, \mathbb I})
+2L_rD_{r, \mathbb I}^{(r-2)/(r-1)}.
\end{align*}
Now we can assume that $D_{r, \mathbb I}\ge 1$,  since otherwise there is nothing to prove. We then obtain $D_{r,\mathbb I}\le (3+2L_r)^{r-1}$ and inequality \eqref{eq:3} follows. This completes the proof of Theorem \ref{thm:main}. 
\end{proof}

\section{Dualization of variational and jumps inequalities}
\label{sec:dvj}
As in the introduction, let $(X,\mathcal B(X),\mu)$ be a $\sigma$-finite measure space. For $p,r\in(1,\infty)$, let $p'=p/(p-1)$ and $r'=r/(r-1)$ denote the conjugate exponents. Let $\mathbb I=(\mathbb I, <)$ be a finite totally ordered index set with at least two elements. We will show that the $r$-variational seminorm defined in \eqref{eq:6}, and more surprisingly the jump-counting function defined in \eqref{eq:7}, can be dualized in a manner analogous to the maximal function identity \eqref{eq:2}.
\subsection{Dualization of variations} Under the above assumptions, define
\[
\mathfrak S(\mathbb I):=\big\{\sigma=(\sigma_j)_{j\in \mathbb N_{\le J_{\sigma}}}\subseteq \mathbb I: \sigma_0<\sigma_1<\ldots<\sigma_{J_\sigma} \text{ for some } J_{\sigma}\in [|\mathbb I|]\big\}.
\]
For a family of measurable functions $F=(F_i)_{i\in\mathbb I}$ the $r$-variational seminorm is defined by
\begin{align}
\label{eq:6}
V^r(F_i: i\in \mathbb I)
:=
\sup_{\sigma\in\mathfrak S(\mathbb I)}
\bigg(
\sum_{j\in[J_{\sigma}]}
|F_{\sigma_j}-F_{\sigma_{j-1}}|^r
\bigg)^{1/r}.
\end{align}
For a family of coefficient functions
$g=(g_{\sigma,j}: \sigma \in \mathfrak S(\mathbb I),\ j\in[J_{\sigma}])$, write
$$
G_g
:=
\sum_{\sigma\in\mathfrak S(\mathbb I)}
\bigg(
\sum_{j\in[J_{\sigma}]}|g_{\sigma,j}|^{r'}
\bigg)^{1/r'}.
$$
By duality, we have
\begin{equation}\label{eq:variation-exact-duality}
\|V^r(F_i: i\in \mathbb I)\|_{L^p(X)}
=
\sup_{\|G_g\|_{L^{p'}(X)}\leq1}
\bigg|
\sum_{\sigma\in\mathfrak S(\mathbb I)}
\sum_{j\in[J_{\sigma}]}
\int_X
(F_{\sigma_j}-F_{\sigma_{j-1}})g_{\sigma,j}\,d\mu
\bigg|.
\end{equation}
Consequently, if $S_i:L^p(X)\to L^p(X)$ is a bounded linear
operator for every $i\in\mathbb I$, and
\begin{equation}\label{eq:variation-vector-adjoint}
\bigg\|
\sum_{\sigma\in\mathfrak S(\mathbb I)}
\sum_{j\in[J_{\sigma}]}
(S_{\sigma_j}^*-S_{\sigma_{j-1}}^*)g_{\sigma,j}
\bigg\|_{L^{p'}(X)}
\leq C\|G_g\|_{L^{p'}(X)},
\end{equation}
then
\begin{equation}\label{eq:variation-from-vector}
\|V^r(S_if:i\in\mathbb I)\|_{L^p(X)}
\leq C\|f\|_{L^p(X)}.
\end{equation}
In fact, the optimal constants in
\eqref{eq:variation-vector-adjoint} and
\eqref{eq:variation-from-vector} are equal, since
\begin{align}
\label{eq:variation-equal-norms}
\sup_{\|f\|_{L^p(X)}\leq1}
\|V^r(S_if:i\in\mathbb I)\|_{L^p(X)}
=
\sup_{\|G_g\|_{L^{p'}(X)}\leq1}
\bigg\|
\sum_{\sigma\in\mathfrak S(\mathbb I)}
\sum_{j\in[J_{\sigma}]}
(S_{\sigma_j}^*-S_{\sigma_{j-1}}^*)g_{\sigma,j}
\bigg\|_{L^{p'}(X)}.
\end{align}
Thus the variational operator and the dualized sum have exactly the same norm. In this sense, identity \eqref{eq:variation-equal-norms} is the variational analogue of the maximal function identity \eqref{eq:2}. We now show that the jump-counting function can also be dualized, see identity \eqref{eq:jump-exact-duality} below. This is a priori less obvious, since, unlike the $r$-variational seminorm, it defines only a quasi-norm.

\subsection{Dualization of jumps} Under the above assumptions, define
\[
\mathfrak C(\mathbb I):=\big\{\pi=((\sigma_j, \tau_j))_{j\in [J_\pi]}\subseteq \mathbb I^2: \sigma_1<\tau_1\leq \sigma_2<\tau_2\leq\ldots \leq \sigma_{J_\pi}<\tau _{J_\pi} \text{
for some } J_\pi\in [|\mathbb I|]\big\}.
\]
For measurable functions $F=(F_i)_{i\in\mathbb I}$ and
$\lambda\in\mathbb R_+$, set
$$
E_\pi(F,\lambda)
:=
\bigcap_{j\in[J_{\pi}]}
\{x\in X: |F_{\tau_j}(x)-F_{\sigma_j}(x)|>\lambda\}.
$$
The $\lambda$-jump counting function can be defined as follows
\begin{align}
\label{eq:7}
N_{\lambda}(F_i: i\in \mathbb I) := \max\big\{J_\pi\mathds{1}_{E_\pi(F,\lambda)}: \pi\in\mathfrak C(\mathbb I)\big\},
\end{align}
and $N_{\lambda}(F_i: i\in \mathbb I)=0$ if $E_\pi(F,\lambda)=\emptyset$.

Let $\mathcal W_\lambda(F)$ be the collection of measurable
weights $w=(w_\pi)_{\pi\in \mathfrak C(\mathbb I)}$ satisfying
$$
w_\pi\geq0,
\qquad
w_\pi=0\ \text{ on }\ E_\pi(F,\lambda)^c,
\qquad
\sum_{\pi\in \mathfrak C(\mathbb I)} w_\pi\leq1,
$$
Then the following exact identity holds
\begin{equation}\label{eq:jump-exact-duality}
\big\|\lambda N_{\lambda}(F_i:i\in \mathbb I)^{1/2}\big\|_{L^p(X)}
=
\lambda\sup_{w\in\mathcal W_\lambda(F)}
\bigg\|\Big(\sum_{\pi\in \mathfrak C(\mathbb I)} J_\pi w_\pi\Big)^{1/2}\bigg\|_{L^p(X)}.
\end{equation}

To verify \eqref{eq:jump-exact-duality}, consider for every $u\in\mathbb R_+$ the  weight $w^{(u)}:=(w_\pi^{(u)})_{ \pi\in \mathfrak C(\mathbb I)}$, where
$$
w_\pi^{(u)}
:=
\frac{\mathds{1}_{E_\pi(F,\lambda)}e^{uJ_{\pi}}}
{1+\sum_{\sigma\in \mathfrak C(\mathbb I)}
\mathds{1}_{E_\sigma(F,\lambda)}e^{uJ_\sigma}}, \qquad \pi\in \mathfrak C(\mathbb I).
$$
Observe  that
$$
\lim_{u\to\infty}\sum_{\pi\in \mathfrak C(\mathbb I)} J_\pi w_\pi^{(u)}=
N_{\lambda}(F_i: i\in \mathbb I),
$$
which by Fatou's lemma implies that
\begin{align*}
\big\|N_{\lambda}(F_i:i\in \mathbb I)^{1/2}\big\|_{L^p(X)}
\le
\sup_{w\in\mathcal W_\lambda(F)}
\bigg\|\Big(\sum_{\pi\in \mathfrak C(\mathbb I)} J_\pi w_\pi\Big)^{1/2}\bigg\|_{L^p(X)}.
\end{align*}
Observing that
$\sum_{\pi\in \mathfrak C(\mathbb I)} J_\pi w_\pi\leq N_{\lambda}(F_i:i\in \mathbb I)$
for every $w\in\mathcal W_\lambda(F)$, the reverse inequality follows, and hence \eqref{eq:jump-exact-duality} holds as desired.
Consequently, a family of operators
$S_i:L^p(X)\to L^p(X)$ satisfies
\begin{equation}\label{eq:jump-operator-bound}
\sup_{\lambda\in\mathbb R_+}
\big\|\lambda N_{\lambda}(S_if:i\in\mathbb I)^{1/2}
\big\|_{L^p(X)}
\leq C\|f\|_{L^p(X)}
\end{equation}
if and only if, for every $f\in L^p(X)$, every $\lambda\in\mathbb R_+$,
and every $w\in\mathcal W_\lambda((S_if)_{i\in \mathbb I})$, we have 
\begin{equation}\label{eq:jump-weight-bound}
\lambda\bigg\|\Big(\sum_{\pi\in \mathfrak C(\mathbb I)} J_\pi w_\pi\Big)^{1/2}\bigg\|_{L^p(X)}
\leq C\|f\|_{L^p(X)}.
\end{equation}
The optimal constants in
\eqref{eq:jump-operator-bound} and
\eqref{eq:jump-weight-bound} are equal by identity \eqref{eq:jump-exact-duality}.

\section{Variational and jump inequalities in Hilbert spaces}\label{sec:hil}
As in the previous section, let $(X,\mathcal B(X),\mu)$ be a $\sigma$-finite measure space. All
operators act on the complex Hilbert space $L^2(X)$, with the scalar product
$$
\langle f,g\rangle:=\int_X f\overline g\,d\mu.
$$
For bounded self-adjoint operators, the notation $S\leq T$ means
that $\langle Sf,f\rangle\leq\langle Tf,f\rangle$ for every
$f\in L^2(X)$. Let $(T_n)_{n\in\mathbb Z}$ be a sequence of linear operators on $L^2(X)$ satisfying: 
\begin{enumerate}[itemsep=5pt]
\item[(H1)] $0\leq T_{n+1}\leq T_n\leq I$ for every $n\in\mathbb Z$. This ensures that every $T_n$ is a contraction on $L^2(X)$.
\item[(H2)] Each $T_n$ preserves positivity, that is maps nonnegative functions
to nonnegative functions $\mu$-almost everywhere on $X$. In particular,   $|Th|\le T|h|$ for every $h\in L^2(X)$.
\item[(H3)] If $h\in L^2(X)$ and $0\leq h\leq1$, then
$0\leq T_nh\leq1$ for every $n\in\mathbb Z$.
\end{enumerate}
Condition (H3) is the sub-Markov property and will be only used in the proof of jump inequality. It is formulated without
applying $T_n$ to the constant function $1$, which need not belong
to $L^2(X)$.

The main result of this section is the following theorem.

\begin{theorem}\label{thm:abstract}
Let $(T_n)_{n\in\mathbb Z}$ be a sequence of bounded operators on $L^2(X)$.
Assume that conditions (H1) and (H2) are satisfied. Then, for every $f\in L^2(X)$, we have
\begin{equation}\label{eq:maximal-theorem}
\big\|\sup_{n\in\mathbb Z}|T_nf|\big\|_{L^2(X)}\leq2\|f\|_{L^2(X)}.
\end{equation}
Moreover, for every $r>2$ one has
\begin{equation}\label{eq:variation-theorem}
\|V^r(T_nf:n\in\mathbb Z)\|_{L^2(X)}\leq \sqrt{33}\Big(1-2^{-\frac{r-2}{2(r-1)}}\Big)^{-1}\|f\|_{L^2(X)}.
\end{equation}
If additionally condition (H3) holds, then
\begin{equation}\label{eq:jump-theorem}
\sup_{\lambda\in\mathbb R_+}
\big\|\lambda N_\lambda(T_nf:n\in\mathbb Z)^{1/2}\big\|_{L^2(X)}
\leq\sqrt5\,\|f\|_{L^2(X)}.
\end{equation}
\end{theorem}

An interesting feature of Theorem \ref{thm:abstract} is that no commutativity, convolution representation, smoothness, dilation structure, or endpoint limits for $T_n$ are assumed. The hypotheses are sufficient for the proof, but they are not intended as a characterization of necessary conditions. The variation constant displayed above grows like $(r-2)^{-1}$ as $r\searrow2$. 

In place of condition (H3), it is enough to assume the uniform bound
$$
\|T_nh\|_{L^\infty(X)}\leq K_\infty\|h\|_{L^\infty(X)},
\qquad h\in L^2(X)\cap L^\infty(X).
$$
Then the right-hand side of \eqref{eq:bounded-energy} becomes
$2K_\infty\int_X W\,d\mu$, and the same proof yields the jump constant $\sqrt{1+4K_\infty}$ in \eqref{eq:jump-theorem} in place of $\sqrt{5}$.
Neither this assumption nor condition (H3) is used in the proofs of the maximal or $r$-variation estimates.

Suppose now that $T_n$ are self-adjoint $L^2(X)$ contractions that preserve positivity and satisfy $T_{n+1}\leq T_n$, without assuming $T_n\geq0$. Then $S_n=\frac{1}{2}(I+T_n)$ satisfies conditions (H1) and (H2). If $T_n$ satisfies (H3), then so does $S_n$. Since
$T_j-T_i=2(S_j-S_i)$, the variation and jump estimates remain valid, with constants
$2\sqrt{33}\big(1-2^{-\frac{r-2}{2(r-1)}}\big)^{-1}$
and $2\sqrt5$, respectively. This further illustrates why the hypotheses of Theorem \ref{thm:abstract} are sufficient and not necessary.

\smallskip

Before proving Theorem \ref{thm:abstract}, we discuss several examples that easily follow from it.

\smallskip

\begin{enumerate}[label={\arabic*.}, itemsep=5pt]

\item (\textit{General dilation groups}). Let $A$ be a $d\times d$
real matrix whose eigenvalues have positive real parts. For
$t\in\mathbb R_+$, define $t^A:=\exp(A\log t)$ and set
$Q:=\operatorname{tr}A$, so that $\det(t^A)=t^Q$. We write $A^*$ for
the adjoint matrix of $A$. Let $\mu$ be a probability measure on ${\mathbb R}^d$. For $(t_n)_{n\in\mathbb Z} \subseteq {\mathbb R}_{+}$, define the dilate $\mu_{t_n}$ via
$$
{\widehat{\mu_{t_n}}}(\xi):= {\widehat{\mu}}(t_n^{A^*}\xi),
\qquad \text{and}\qquad
T_nf:=\mu_{t_n} * f.
$$
The family of operators $(T_n)_{n\in\mathbb Z}$ satisfies hypotheses (H2) and (H3) but not necessary (H1).

In \cite{JSW}, variation and jump inequalities were established for {\it lacunary} dilates $T_n f = \mu_{2^n} * f$ for $n\in\mathbb Z$ under the additional hypothesis
\begin{align}
\label{F-decay} 
|{\widehat \mu}(\xi)| \ \le \ (1 + |\xi|)^{-\beta} \quad \text{for some} \quad \beta\in\mathbb R_+.
\end{align}
The proof involves, as usual, identifying a martingale structure adapted to the dilation group $(t^A)_{t\in\mathbb R_+}$ which is highly nontrivial for general matrices $A$.

Let us replace the decay hypothesis \eqref{F-decay} with the following monotoncity hypothesis:
\begin{align}
\label{eq:8}
1\geq\widehat\mu(t_n^{A^*}\xi)
\geq\widehat\mu(t_{n+1}^{A^*}\xi)\geq0.
\end{align}
Then Theorem \ref{thm:abstract} applies. Interestingly, in this case
the constants are independent of the dimension, the dilation group and
the measure, subject only to the stated
assumptions. Importantly,  lacunarity of the dilation parameters $(t_n)_{n\in\mathbb Z}$ is not needed, the 
conclusion holds for any increasing sequence of scales satisfying
\eqref{eq:8}. The argument uses only
positivity preservation, the contraction property, and the ordering of
the Fourier multipliers.

\item (\textit{Verification of condition \eqref{eq:8}}). A natural
question is for which measures condition \eqref{eq:8} holds. 
We consider the case of {\it quasi-radial} multipliers; that is, when ${\widehat \mu}(\xi) = m(\|\xi\|)$ for some nonincreasing multiplier $m:[0,\infty) \to [0,1]$. Here $\|\cdot\|$ is a norm given by $\|x\|^2 = \langle Bx, x\rangle$ where $B$ is a positive definite, symmetric matrix with the property that $t \to \|t^{A^{*}} x\|$ is a strictly increasing function for every nonzero $x \in {\mathbb R}^d$. In this case, we see that \eqref{eq:8} holds since $t \to {\widehat \mu}(t^{A^{*}} \xi) = m(\|t^{A^{*}}\xi\|)$ is nonincreasing.

If $A=\operatorname{diag}(a_1,\ldots,a_d)$ is a diagonal matrix with
$a_1,\ldots,a_d\in\mathbb R_+$, then the usual euclidean norm $\|\cdot\|$ (so $B = I$ is identity matrix) works. For general real matrices $A$ whose eigenvalues have positive real parts, the positive symmetric matrix
$$
B:=\int_0^\infty e^{-uA}e^{-uA^*}\,du
$$
does the job. See \cite[Proposition 1-7]{SW} for details.

\item (\textit{Positive Radon measures on locally compact abelian groups}). More generally, let $\mathbb G=(\mathbb G,+)$ be a locally compact abelian group as in Section \ref{sec:radon}. Let $(\nu_n)_{n\in\mathbb Z}$ be a sequence of probability Radon measures on $\mathbb G$ satisfying
\[
1\geq\widehat\nu_n(\chi)
\geq\widehat\nu_{n+1}(\chi)\geq0,
\qquad
n\in\mathbb Z,\quad \chi\in\widehat{\mathbb G}.
\]
For $f\in L^2(\mathbb G)$, define $T_nf:=f*\nu_n.$
Then the operators $T_n$ satisfy conditions (H1), (H2), and (H3), so Theorem \ref{thm:abstract} applies. No density, smoothness, group geometry, or underlying dilation structure is required. The case $\mathbb G=\mathbb R^d$ is recovered as a particular example.

\item (\textit{Martingales}).  Let $(\mathcal F_n)_{n\in\mathbb Z}$ be a
decreasing filtration and let
$\mathbb E_n$ be the corresponding conditional expectation.  These operators are
positive self-adjoint contractions, satisfy $\mathbb E_n1=1$, and are nested
orthogonal projections.  Consequently, they satisfy conditions (H1), (H2),
and (H3) and Theorem \ref{thm:abstract} gives the corresponding
maximal, $r$-variational, and jump estimates on $L^2$. We remark that for martingales, the constant $2$ in the maximal inequality \eqref{eq:maximal-theorem} is optimal. Furthermore,
Theorem \ref{thm:abstract} exploits the Hilbert-space mechanism
underlying martingale estimates, without using martingale difference
decompositions. Until this work, $r$-variational and jump estimates
were widely viewed as being intrinsically tied to martingale methods.

\item (\textit{Sub-Markov semigroups}).  Let $(S_t)_{t\in \mathbb R_+}$ be a strongly continuous symmetric sub-Markov semigroup on $L^2(X)$. Here symmetry means that each $S_t,$  for $t\in \mathbb R_+,$ is self-adjoint, while being sub-Markov means that $S_t$ is a contraction on $L^2(X)$ and that (H2), (H3) hold (with $T_n$ replaced by $S_t$). Then, for a non-decreasing sequence of positive times $(t_n)_{n\in \Z}\subseteq \mathbb R_+,$ the operators $T_n=S_{t_n}$ satisfy the conditions (H1), (H2), and (H3). The last two conditions are obvious from the definition while (H1) follows from the contractivity and semigroup properties since $0\le \langle S_{t_{n+1}/2} f,  S_{t_{n+1}/2}f\rangle=\langle S_{t_{n+1}} f,  f\rangle \le \langle f, f\rangle$ and
\begin{align*}\langle S_{t_{n+1}} f,  f\rangle&=\langle S_{t_{n+1}/2} f,  S_{t_{n+1}/2}f\rangle=\|S_{t_{n+1}/2}f\|^2_{L^2(X)}=\|S_{(t_{n+1}-t_n)/2}\,S_{t_n/2}f\|^2_{L^2(X)}\\
&\le \|S_{t_{n}/2}f\|^2_{L^2(X)}=\langle S_{t_n/2} f,  S_{t_n/2}f\rangle=\langle S_{t_n} f,  f\rangle. 
\end{align*}
Therefore, Theorem \ref{thm:abstract} applies to the operators $S_{t_n}.$ Furthermore, because of the strong $L^2$ continuity, a density argument allows us to obtain the maximal \eqref{eq:maximal-theorem}, variational \eqref{eq:variation-theorem} and jump \eqref{eq:jump-theorem} estimates for the full (continuous) semigroup $(S_t)_{t\in \mathbb R_+}.$ We remark that the constant $2$ in \eqref{eq:maximal-theorem} is also optimal for sub-Markov semigroups, see \cite{CHMNW}. 

\item (\textit{$L^p$ and endpoint extensions}).  For Schwartz bump
families associated with dilations, the $L^2$ estimates
above can be upgraded, under the usual uniform kernel size and
smoothness assumptions, to strong $L^p$ estimates for $1<p<\infty$ and
to weak-type $(1,1)$ estimates by the standard vector-valued
Calder\'on--Zygmund decomposition and interpolation arguments.  For
conditional expectations, analogous conclusions are already known;
see, for instance, the appendix in \cite{MST2}. Thus, the genuinely
new point here is the direct treatment of $r$-variational and jump
estimates beyond the martingale setting, using purely Hilbert-space
methods.
\end{enumerate}

\begin{proof}[Proof of Theorem \ref{thm:abstract}]
It suffices to establish inequalities \eqref{eq:maximal-theorem},
\eqref{eq:variation-theorem}, and \eqref{eq:jump-theorem} with
$[\pm N]$ in place of $\mathbb Z$ for every $N\in\mathbb Z_+$. The
stated inequalities then follow by the monotone convergence
theorem. The proof is divided into several steps for clarity. The key
inputs are the dualization identities \eqref{eq:variation-equal-norms}
and \eqref{eq:jump-exact-duality} established in the previous section.

\smallskip

\paragraph{\bf Step 1} We begin with a variant of the Cauchy--Schwarz inequality for bounded positive  operators on a Hilbert space $\mathbb H$ equipped with an inner product $\langle\cdot,\cdot\rangle:\mathbb H\times\mathbb H\to\mathbb C$ and the induced norm $\|\cdot\|$. Fix $M\in\mathbb Z_+$ and let $D_1,\ldots,D_M$ be bounded positive  operators
on $\mathbb H$ (not necessarily commuting) such that $\sum_{k\in[M]}D_k\leq I$. Then for any  $h_1,\ldots, h_M\in\mathbb H$, we have 
\begin{equation}\label{eq:operator-cs}
\Big\|\sum_{k\in[M]} D_kh_k\Big\|^2
\leq\sum_{k\in[M]}\langle D_kh_k,h_k\rangle.
\end{equation}
Indeed, since each $D_k$ is bounded and positive, $\langle u, v\rangle_k := \langle D_k u, v \rangle$ defines a semi-inner product and so 
for any $u\in\mathbb H$, the Cauchy-Schwarz inequality implies
$|\langle D_kh_k,u\rangle|^2 \leq\langle D_kh_k,h_k\rangle\langle D_ku,u\rangle$ for $k\in[M]$.  Using the linearity of the inner product, followed by the above inequality, the Cauchy--Schwarz inequality and $\sum_{k\in[M]}D_k\leq I$, we obtain
\begin{align*}
\Big|\Big\langle\sum_{k\in[M]}D_kh_k,u\Big\rangle\Big|
&\leq\sum_{k\in[M]}\langle D_kh_k,h_k\rangle^{1/2}
                 \langle D_ku,u\rangle^{1/2}\\
&\leq\Big(\sum_{k\in[M]}\langle D_kh_k,h_k\rangle\Big)^{1/2}
       \Big\langle\sum_{k\in[M]}D_ku,u\Big\rangle^{1/2}\\
&\leq\Big(\sum_{k\in[M]}\langle D_kh_k,h_k\rangle\Big)^{1/2}\|u\|.
\end{align*}
Taking the
supremum over $\|u\|\leq1$ inequality \eqref{eq:operator-cs} follows as claimed.

For further reference,  for every self-adjoint bounded  operator $T:L^2(X)\to L^2(X)$, we have 
\begin{equation}\label{eq:quadratic-identity}
\langle Tu,u\rangle-\langle Tv,v\rangle
=\operatorname{Re}\langle u-v,T(u+v)\rangle.
\end{equation}
If additionally $T$ preserves positivity, then $|Th|\le T|h|$ for
every $h\in L^2(X)$, and consequently  for every $u, v, H\in L^2(X)$ such that
$|u|,|v|\leq H$, using identity \eqref{eq:quadratic-identity} we obtain
\begin{equation}\label{eq:quadratic-bound}
\left|\langle Tu,u\rangle-\langle Tv,v\rangle\right|
\leq2\int_X |u-v|\,TH\,d\mu.
\end{equation}

\paragraph{\bf Step 2}
We now prove maximal estimate from \eqref{eq:maximal-theorem}.
Define
$$
B_N:=\sup_{\|f\|_{L^2(X)}\leq1}
\big\|\sup_{i\in [\pm N]}|T_if|\big\|_{L^2(X)}.
$$
Note that $B_N\leq\sqrt{2N+1}<\infty$, since every $T_i$ is a contraction.
By dualization identity \eqref{eq:2} and
self-adjointness, we may write
\begin{equation}\label{eq:maximal-dual}
B_N=\sup_{\|G\|_{L^2(X)}\leq1}
\Big\|\sum_{i\in[\pm N]}T_ig_i\Big\|_{L^2(X)},
\quad \text{ where } \quad G=\sum_{i\in[\pm N]}|g_i|.
\end{equation}
Set $S_i:=\sum_{j=-N}^ig_j$ and $S_{-N-1}:=0$. Summation by parts gives
$$
\sum_{i\in[\pm N]}T_ig_i
=\sum_{i=-N}^{N-1}(T_i-T_{i+1})S_i+T_NS_N.
$$
We have  $T_i-T_{i+1}\ge0$ and $T_N\ge 0$  and  $\sum_{i=-N}^{N-1}(T_i-T_{i+1})+T_N=T_{-N}\leq I$.
Therefore, inequality \eqref{eq:operator-cs} and telescoping  give
\begin{align*}
\Big\|\sum_{i\in[\pm N]}T_ig_i\Big\|_{L^2(X)}^2
&\leq\sum_{i=-N}^{N-1}\langle(T_i-T_{i+1})S_i,S_i\rangle
       +\langle T_NS_N,S_N\rangle\\
&=\sum_{i\in[\pm N]}\bigl(
\langle T_iS_i,S_i\rangle-\langle T_iS_{i-1},S_{i-1}\rangle
\bigr)\\
&\leq2\sum_{i\in[\pm N]}\int_X|g_i|T_iG\,d\mu\\
&\leq2\int_XG\sup_{ i\in [\pm N]}T_iG\,d\mu
\leq2B_N\|G\|_{L^2(X)}^2.
\end{align*}
In the third line, we used estimate \eqref{eq:quadratic-bound}, together with the identities $S_i-S_{i-1}=g_i$ and the bounds $|S_i|,|S_{i-1}|\leq G$. In the final inequality, we applied the Cauchy--Schwarz inequallity, and used the definition of the constant $B_N$. Hence we obtain $B_N^2\leq 2B_N$ or $B_N\leq 2$. Letting $N\to\infty$, we conclude inequality  \eqref{eq:maximal-theorem} as desired.

\smallskip

\paragraph{\bf Step 3}
We now refine the argument from the previous step so that it can be applied to the variational and jump estimates.
Set $D_k:=T_{k-1}-T_k$ for $k\in [\pm N]\setminus\{-N\}$ and note that $\sum_{k=-N+1}^ND_k=T_{-N}-T_N\leq I$,  since the operators are monotone.  For $k\in [\pm N]\setminus\{-N\}$ suppose that $h_k\in L^2(X)$ satisfies $|h_k|\leq H$ for some nonnegative
$H\in L^2(X)$, and define
$$
W:=\sum_{k=-N+1}^{N-1}|h_{k+1}-h_k|.
$$
Then we claim that
\begin{equation}\label{eq:energy}
\Big\|\sum_{k=-N+1}^ND_kh_k\Big\|_{L^2(X)}^2
\leq\|H\|_{L^2(X)}^2+4\|H\|_{L^2(X)}\|W\|_{L^2(X)}.
\end{equation}
If additionally condition (H3) holds and $|h_k|\leq1$, then also
\begin{equation}\label{eq:bounded-energy}
\Big\|\sum_{k=-N+1}^ND_kh_k\Big\|_{L^2(X)}^2
\leq\|h_{-N+1}\|_{L^2(X)}^2+2\int_XW\,d\mu,
\end{equation}
whenever the right-hand side is finite.

To prove estimate \eqref{eq:energy}, we apply inequality \eqref{eq:operator-cs}, which gives

\begin{align}
\Big\|\sum_{k=-N+1}^ND_kh_k\Big\|_{L^2(X)}^2
&\leq\sum_{k=-N+1}^N\langle D_kh_k,h_k\rangle\notag\\
&=\langle T_{-N}h_{-N+1},h_{-N+1}\rangle
 -\langle T_Nh_N,h_N\rangle\notag\\
&\quad+\sum_{k=-N+1}^{N-1}\bigl(
\langle T_kh_{k+1},h_{k+1}\rangle
-\langle T_kh_k,h_k\rangle\bigr).
\label{eq:energy-identity}
\end{align}
The first term is at most
$\|h_{-N+1}\|_{L^2(X)}^2\leq\|H\|_{L^2(X)}^2$; the second is
nonpositive and can be discarded. By \eqref{eq:quadratic-bound},
followed by the Cauchy--Schwarz inequality and the maximal estimate
\eqref{eq:maximal-theorem}, the remaining terms are bounded by
$$
2\int_X\sum_{k=-N+1}^{N-1}|h_{k+1}-h_k|T_kH\,d\mu
\leq2\int_XW\max_{k\in [\pm N]}T_kH\,d\mu
\leq4\|W\|_{L^2(X)}\|H\|_{L^2(X)}.
$$
This proves \eqref{eq:energy}. For \eqref{eq:bounded-energy}, use (H3) and
 obtain $|T_kh_j|\leq T_k|h_j|\le 1$.
Identity \eqref{eq:quadratic-identity} then bounds each
quadratic increment by $2\int_X|h_{k+1}-h_k|\,d\mu$.
Substitution into \eqref{eq:energy-identity} gives the assertion.

\smallskip

\paragraph{\bf Step 4}
Fix $r>2$, we now prove inequality \eqref{eq:variation-theorem} with $[\pm N]$ in
place of $\mathbb Z$ for every $N\in\mathbb Z_+$.  By dualization
formula \eqref{eq:variation-equal-norms} for $r$-variations \eqref{eq:6} and
self-adjointness, it suffices to prove that
\begin{equation}\label{eq:dual-variation}
\|\Delta\|_{L^2(X)}\leq \sqrt{33}\Big(1-2^{-\frac{r-2}{2(r-1)}}\Big)^{-1}\|G\|_{L^2(X)},
\end{equation}
where for any coefficient functions $(g_{\sigma,j}: \sigma \in \mathfrak S([\pm N]),\ j\in[J_{\sigma}])$, we have
\[
\Delta:=\sum_{\sigma\in\mathfrak S([\pm N])}\sum_{j\in[J_\sigma]}(T_{\sigma_j}-T_{\sigma_{j-1}})g_{\sigma,j},
\qquad
G_\sigma:=\Big(\sum_{j\in[J_\sigma]}|g_{\sigma,j}|^{r'}\Big)^{1/r'},
\qquad G:=\sum_{\sigma\in\mathfrak S([\pm N])} G_\sigma.
\]

For every integer $m \in\mathbb N$, set $g_{\sigma,j}^{(m )}: =g_{\sigma,j}\mathds{1}_{\{2^{-m -1}G_\sigma<|g_{\sigma,j}| \leq2^{-m }G_\sigma\}}$,
and define
$$
\Delta_m :=\sum_{\sigma\in\mathfrak S([\pm N])}\sum_{j\in[J_\sigma]}(T_{\sigma_j}-T_{\sigma_{j-1}})g_{\sigma,j}^{(m )}.
$$
Since $T_{\sigma_j}-T_{\sigma_{j-1}}
=-\sum_{k=\sigma_{j-1}+1}^{\sigma_j}D_k$, where $D_k:=T_{k-1}-T_k$,
we note that
\[
\Delta_m =-\sum_{k=-N+1}^ND_kh_k^{(m )},
\]
where for each $k\in [\pm N]\setminus\{-N\}$, we have
$$
h_k^{(m )}:= \sum_{\sigma\in\mathfrak S([\pm N])}\sum_{\substack{j\in[J_\sigma]\\ \sigma_{j-1}<k\leq \sigma_j}}
g_{\sigma,j}^{(m )}.
$$
Observe that for every fixed $\sigma\in\mathfrak S([\pm N])$ the intervals $(\sigma_{j-1},\sigma_j]$  are disjoint when $j\in[J_\sigma]$. Hence the inner sum in the definition of $h_k^{(m )}$ contains at most one nonzero summand, and consequently\begin{equation}\label{eq:level-height}
|h_k^{(m )}|\leq2^{-m }\sum_{\sigma\in\mathfrak S([\pm N])} G_\sigma=2^{-m }G.
\end{equation}
In particular, 
$$
\|\Delta_m \|_{L^2(X)} \ \le \ 4N \, \sup_{k\in [\pm N]}  \|h_k^{(m)}\|_{L^2(X)} \ \le \ 4 N 2^{-m} \|G\|_{L^2(X)}
$$
and so
\begin{equation}\label{eq:level-sum}
\Delta=\sum_{m \in\mathbb N} \Delta_m \quad\text{ in } \quad L^2(X).
\end{equation}

For $x\in X$ define $\mathbb J_{\sigma, m }(x):=\{j\in[J_\sigma]: |g_{\sigma,j}^{(m )}(x)|>0\}$, and note that $|\mathbb J_{\sigma, m }(x)|\le 2^{(m +1)r'}$, since 
\[
G_\sigma(x)^{r'}\ge \sum_{j\in \mathbb J_{\sigma, m }(x)}|g_{\sigma,j}^{(m)}(x)|^{r'}> |\mathbb J_{\sigma, m }(x)|2^{-(m +1)r'}G_\sigma(x)^{r'}.
\]
Observing that for each $\sigma\in\mathfrak S([\pm N])$ and $j\in[J_\sigma]$, we have
$$
\sum_{k=-N+1}^{N-1} 
|\mathds{1}_{(\sigma_{j-1},\sigma_j]}(k+1) - \mathds{1}_{(\sigma_{j-1},\sigma_j]}(k)| \le 2.
$$
Now using the triangle inequality
and the preceding count, we conclude that
\begin{align}
\begin{split}
W_m 
:=&\sum_{k=-N+1}^{N-1}|h_{k+1}^{(m )}-h_k^{(m )}|
\leq2\sum_{\sigma\in\mathfrak S([\pm N])}\sum_{j\in[J_\sigma]}|g_{\sigma,j}^{(m )}|\\
\leq&2\sum_{\sigma\in\mathfrak S([\pm N])}2^{(m +1)r'}2^{-m }G_\sigma
=2^{r'+1}2^{m (r'-1)}G.
\end{split}
\label{eq:level-total-variation}
\end{align}
Taking $H_{m }:=2^{-m }G$  we have  $|h_k^{(m )}|\leq H_m $ by \eqref{eq:level-height}.
Applying inequality \eqref{eq:energy} with $H=H_m $ and $W=W_m $, we immediately conclude 
\begin{align*}
\|\Delta_m \|_{L^2(X)}^2
\leq\left(2^{-2m }
       +2^{r'+3}2^{-m (2-r')}\right)\|G\|_{L^2(X)}^2
\leq 33\cdot 2^{-m (2-r')}\|G\|_{L^2(X)}^2,
\end{align*}
since $r'\in(1,2)$. Taking square roots, using identity
\eqref{eq:level-sum}, and summing the resulting geometric series, we
obtain \eqref{eq:dual-variation} with the desired constant. Letting
$N\to\infty$ then yields estimate \eqref{eq:variation-theorem}.  \smallskip

\paragraph{\bf Step 5}
Assume condition (H3), and  fix $\lambda\in\mathbb R_+$.  We now prove inequality \eqref{eq:jump-theorem} with
$[\pm N]$ in place of $\mathbb Z$ for every $N\in\mathbb Z_+$. We will use dualization identity from \eqref{eq:jump-exact-duality}, which for $p=2$ takes particularly simple form
\begin{align}
\label{eq:5}
\big\|\lambda N_{\lambda}(T_nf:n\in [\pm N])^{1/2}\big\|_{L^2(X)}^2
=
\lambda^2\sup_{w\in\mathcal W_\lambda((T_nf)_{n\in [\pm N]})}
\int_X\sum_{\pi\in \mathfrak C([\pm N])} J_\pi w_\pi\, d\mu.
\end{align}
Fix $w\in\mathcal W_\lambda((T_nf)_{n\in [\pm N]})$ and  define 
\[
L_{w}:=\int_X\sum_{\pi\in \mathfrak C([\pm N])} J_\pi w_\pi\, d\mu.
\]
For $\pi\in \mathfrak C([\pm N])$ and $j\in[J_{\pi}]$ set $b_{\pi,j}:=(T_{\sigma_j}-T_{\tau_j})f$. Now for each $k\in [\pm N]\setminus\{-N\}$, define
\[
h_k:=\sum_{\sigma\in\mathfrak S([\pm N])}\sum_{\substack{j\in[J_\sigma]\\ \sigma_{j}<k\leq \tau_j}}g_{\pi,j}, \qquad \text{ where } \qquad g_{\pi,j}:=w_\pi\frac{b_{\pi,j}}{|b_{\pi,j}|}.
\]
Observe that $g_{\pi,j}$ is well defined, since $|b_{\pi,j}|>\lambda$ on the support of $w_\pi$, and we set $g_{\pi,j}=0$ outside the support of $w_\pi$. All these functions belong to $L^2(X)$,
because the weights are bounded by one and they are supported in a
fixed set $E_{\pi}((T_nf)_{n\in [\pm N]}, \lambda)$ which has  finite measure. Observe that
\begin{align}
\lambda L_w
&\leq\int_X\sum_{\pi\in \mathfrak C([\pm N])} \sum_{j\in[J_{\pi}]}w_\pi|b_{\pi,j}|\,d\mu
=\sum_{\pi\in \mathfrak C([\pm N])} \sum_{j\in[J_{\pi}]}\langle b_{\pi,j},g_{\pi,j}\rangle
=\Big\langle f,\sum_{k=-N+1}^ND_kh_k\Big\rangle.
\label{eq:jump-pairing}
\end{align}
The last equality uses self-adjointness and
telescoping $T_{\sigma_j}-T_{\tau_j}=\sum_{k=\sigma_j+1}^{\tau_j}D_k$.
As in the previous step for any fixed $\pi\in\mathfrak S([\pm N])$ the intervals $(\sigma_{j},\tau_j]$  are disjoint when $j\in[J_\pi]$, hence 
$$
|h_k|\leq\sum_{\pi\in \mathfrak C([\pm N])} w_\pi\leq1.
$$
Moreover, each sequence $[\pm N]\setminus\{-N\}\ni k\mapsto\mathds{1}_{(\sigma_{j},\tau_j]}(k)$ has
discrete total variation at most two, using the triangle inequality, we obtain
$$
\int_X\sum_{k=-N+1}^{N-1}|h_{k+1}-h_k|\,d\mu
\leq2\int_X\sum_{\pi\in \mathfrak C([\pm N])} \sum_{j\in[J_{\pi}]}|g_{\pi,j}|\,d\mu=2L_w.
$$
Also we have $\|h_{-N+1}\|_{L^2(X)}^2\leq L_w$, since
$$
|h_{-N+1}|^2\leq|h_{-N+1}|
\leq\sum_{\pi\in \mathfrak C([\pm N])} w_\pi\leq\sum_{\pi\in \mathfrak C([\pm N])} J_\pi w_\pi.
$$
Applying inequality 
\eqref{eq:bounded-energy}, we conclude that
$$
\Big\|\sum_{k=-N+1}^ND_kh_k\Big\|_{L^2(X)}^2\leq5L_w.
$$
Thus \eqref{eq:jump-pairing} and Cauchy--Schwarz imply $\lambda L_w\leq\sqrt5\,\|f\|_{L^2(X)}L_w^{1/2}$.
If $L_w=0$, there is nothing to prove. Otherwise, dividing by $L_w^{1/2}$ and squaring, we immediately obtain
$\lambda^2L_w\leq5\|f\|_{L^2(X)}^2$. Taking the supremum over all weights $w\in\mathcal W_\lambda((T_nf)_{n\in[\pm N]})$, invoking \eqref{eq:5}, and then letting $N\to\infty$, we obtain \eqref{eq:jump-theorem}. This completes the proof of Theorem \ref{thm:abstract}.
\end{proof}

\end{document}